\documentclass[12pt,reqno,twoside]{amsart}
\usepackage{graphicx}
\usepackage{amssymb}
\usepackage{epstopdf}
\usepackage[asymmetric,top=3.5cm,bottom=4.3cm,left=3.1cm,right=3.1cm]{geometry}
\usepackage{booktabs} 
\usepackage{array} 
\usepackage{paralist} 
\usepackage{verbatim} 
\usepackage{subfig} 
\usepackage{tabularx}
\usepackage{amsmath,amsfonts,amsthm,mathrsfs,amssymb,cite}
\usepackage[usenames]{color}
\usepackage{bm}

\newtheorem{thm}{Theorem}[section]

\newtheorem{lem}{Lemma}[section]

\theoremstyle{definition}
\newtheorem{defn}{Definition}[section]
\theoremstyle{remark}

\newtheorem{rem}{Remark}[section]
\numberwithin{equation}{section}

\newcommand{\bmm}[1]{\mbox{\boldmath{$#1$}}}
\newcommand{\cl}{\mathcal{L}}
\newcommand{\bu}{\mathbf{u}}
\newcommand{\bv}{\mathbf{v}}
\newcommand{\bV}{\mathbf{V}}
\newcommand{\bw}{\mathbf{w}}
\newcommand{\bff}{\mathbf{f}}
\newcommand{\curl}{\mbox{curl}}

\title[Anomalous Localized Resonance in elastostatics]{On Anomalous Localized Resonance for the Elastostatic System}

\author{Hongjie Li}
\address{Department of Mathematics, Hong Kong Baptist University, Kowloon Tong, Hong Kong SAR.\vspace*{-4mm}}
\address{\vspace*{-4mm}and}
\address{HKBU Institute of Research and Continuing Education, Virtual University Park, Shenzhen, P. R. China.}
\email{hongjie$_{-}$li@yeah.net}

\author{Hongyu Liu}
\address{Department of Mathematics, Hong Kong Baptist University, Kowloon Tong, Hong Kong SAR.\vspace*{-4mm}}
\address{\vspace*{-4mm}and}
\address{HKBU Institute of Research and Continuing Education, Virtual University Park, Shenzhen, P. R. China.}
\email{hongyu.liuip@gmail.com; hongyuliu@hkbu.edu.hk}

\begin{document}
\maketitle

\begin{abstract}
We consider the anomalous localized resonance due to a plasmonic structure for the elastostatic system in $\mathbb{R}^2$. The plasmonic structure takes a general core-shell-matrix form with the metamaterial located in the shell. If there is no core, we show that resonance occurs for a very broad class of sources. If the core is nonempty and of an arbitrary shape, we show that there exists a critical radius such that anomalous localized resonance (ALR) occurs. Our argument is based on a variational technique by making use of the primal and dual variational principles for the elastostatic system, along with the construction of suitable test functions.

\medskip

\medskip



\end{abstract}

\section{Introduction}

Recently, there is growing interest in studying the resonance phenomena for materials with negative parameters and their connection to invisibility cloaking. The mathematical principle lies in that the ellipticity of the governing PDE system is lost in the limiting case as the loss parameter approaching zero, and hence resonance occurs for a suitable forcing term at or near the resonant frequency. Moreover, the resonance strongly depends on the location of the source. Hence, the term ``anomalous localized resonance" is used in the literature, and those negative materials are referred to as ``plasmonic materials".

The anomalous localized resonance (ALR) has been extensively investigated for optical parameter distributions; respectively, modelled by the Laplace equation \cite{Acm13,Ack13,Ack14,AK,Bos10,Brl07,CKKL,Klsap,LLL,GWM1,GWM2,GWM3,GWM4,GWM6,GWM7,GWM8,GWM9}, the Helmholtz equation \cite{ADM,AMRZ,AKL,KLO}, and the Maxwell system \cite{ARYZ}. In the literature, there are two approaches that have been proposed in analyzing the resonance behaviors: one is based on the analysis of the spectral properties of the Neumann-Poincar\'e operator via the layer potential theory \cite{Ack13} and the other one is based on variational arguments via the use of primal and dual variational principles for the corresponding PDE systems \cite{Klsap}. The spectral approach initiated in \cite{Ack13} by Ammari et al. for the Laplace equation can yield a very accurate characterization of the anomalous localized resonance as well as its connection to the invisibility cloaking effects. The spectral approach has been extensively used in characterizing the anomalous localized resonances (ALR) and the cloaking by ALR (CALR) for the Laplace equation and the Helmholtz system \cite{Acm13,Ack14,AK,AKL,CKKL}. However, the spectral approach requires accurate spectral information of the Neumann-Poincar\'e operator and is mainly restricted to spherical and elliptical geometries. The variational approach initiated in \cite{Klsap} by Kohn et al. for the Laplace equation can deal with general geometries, but can only be used to show the ALR results. The variational approach was followed in \cite{LLL} to show the ALR results for the Laplace equation in three dimensions.

In a very recent work by Ando et al. \cite{AKKY}, the spectral theory of Neumann-Poincar\'e operator was extended from the Laplace equation to the Lam\'e system governing the elastostatics. Using the spectral approach, the authors also established the ALR and CARL results in the spherical and elliptic geometries for the elastostatic system. In this paper, we aim to establish the ALR results for the elastostatic system via the variational approach. In what follows, we first present the mathematical formulation of our study.

 Let $\mathbf{C}(x):=(\mathrm{C}_{ijkl}(x))_{i,j,k,l=1}^N$, $x\in\mathbb{R}^N$ with $N=2,3,$ be a four-rank tensor such that
 \begin{equation}\label{eq:tensor}
 \mathrm{C}_{ijkl}(x):=\lambda(x)\delta_{ij}\delta_{kl}+\mu(x)(\delta_{ik}\delta_{jl}+\delta_{il}\delta_{jk}),\ \ x\in\mathbb{R}^N,
 \end{equation}
 where $\lambda, \mu\in\mathbb{C}$ are complex-valued functions, and $\delta$ is the Kronecker delta. $\mathbf{C}(x)$ describes an isotropic elasticity tensor distributed in the space, where $\lambda$ and $\mu$ are called the Lam\'e constants. For a regular elastic material, it is required that the Lam\'e constants satisfy the strong convexity condition,
 \begin{equation}\label{eq:convex}
 \mu>0\qquad\mbox{and}\qquad N\lambda+2\mu>0.
 \end{equation}
 The existence of exotic elastic materials with negative stiffness was shown in \cite{KM} and \cite{LLBW}, which we shall generally refer to as the plasmonic materials in the current article.
We shall write $\mathbf{C}_{\lambda,\mu}$ to specify the dependence of the elastic tensor on the Lam\'e parameters $\lambda$ and $\mu$. Let $\Sigma$ and $\Omega$ be bounded domains in $\mathbb{R}^N$ with connected Lipschitz boundaries such that $\Sigma\Subset\Omega$. Consider an elastic parameter distribution $\mathbf{C}_{\widetilde{\lambda}, \widetilde{\mu}}$ given with
\begin{equation}\label{eq:tensor1}
\big(\widetilde{\lambda}(x), \widetilde{\mu}(x)\big)=\big(A(x)+\mathrm{i}\delta\big)(\lambda, \mu), \quad x\in\mathbb{R}^N,
\end{equation}
where $\delta\in\mathbb{R}_+$ denotes a loss parameter; $(\lambda, \mu)$ are two constants satisfying the strong convexity condition \eqref{eq:convex}; and $A(x)$ has a matrix-shell-core character in the sense that
\begin{equation}\label{eq:tensor2}
A(x)=\begin{cases}
+1,\qquad & x\in\Sigma,\medskip\\
c,\qquad & x\in\Omega\backslash\overline{\Sigma},\medskip\\
+1,\qquad & x\in\mathbb{R}^N\backslash\overline{\Omega},
\end{cases}
\end{equation}
where $c$ is constant that will be specified later. In principle, $c$ will be negative-valued so that \eqref{eq:tensor1} yields a plasmonic structure.
Let $\mathbf{u}_\delta(x)\in\mathbb{C}^N$, $x\in\mathbb{R}^N$, denote the displacement field in the space that is occupied by the elastic material distribution $\mathbf{C}_{\widetilde{\lambda},\widetilde{\mu}}$ as described above. In the quasi-static regime, $\mathbf{u}_\delta(x)$ verifies the following Lam\'e system
\begin{equation}\label{eq:lame1}
\begin{cases}
&\mathcal{L}_{\widetilde\lambda,\widetilde\mu} \mathbf{u}_\delta(x)=\mathbf{f}(x),\quad x\in\mathbb{R}^N,\medskip\\
&\mathbf{u}_\delta(x)=\mathcal{O}\big(\|x\|^{-1}\big)\quad\mbox{as}\ \ \|x\|\rightarrow+\infty,
\end{cases}
\end{equation}
where $\mathbf{f}$ is an $\mathbb{R}^N$-valued function that is compactly supported outside $\Omega$, signifying a force term. It is required that
\begin{equation}\label{eq:source1}
\int_{\mathbb{R}^N}\mathbf{f}(x)\ dV(x)=0.
\end{equation}
In \eqref{eq:lame1}, the partial differential operator $\mathcal{L}_{\widetilde{\lambda},\widetilde{\mu}}$ is defined as follows,
\begin{equation}\label{eq:lame2}
\mathcal{L}_{\widetilde{\lambda},\widetilde{\mu}} \mathbf{u}_\delta:=\nabla\cdot\mathbf{C}_{\widetilde{\lambda},\widetilde{\mu}} {\nabla}^s\mathbf{u}_\delta=\widetilde{\mu}\Delta\mathbf{u}_\delta+(\widetilde{\lambda}+\widetilde{\mu})\nabla\nabla\cdot\mathbf{u}_\delta,
\end{equation}
where ${\nabla}^s$ denotes the symmetric gradient given by
\[
{\nabla}^s\mathbf{u}_\delta:=\frac 1 2(\nabla\mathbf{u}_\delta+\nabla\mathbf{u}_\delta^T),
\]
and $T$ signifies the matrix transpose.

Next, for $\mathbf{u}\in H^1_{\text{loc}}(\mathbb{R}^N)$ and $\mathbf{v}\in H^1_{\text{loc}}(\mathbb{R}^N)$, we introduce
\begin{equation}\label{eq:energy1}
\mathbf{P}_{\lambda,\mu}(\mathbf{u},\mathbf{v}):=\int_{\mathbb{R}^N}\big[\lambda(\nabla\cdot\mathbf{u})\overline{(\nabla\cdot\mathbf{v})}(x)+2\mu\nabla^s\mathbf{u}:\overline{\nabla^s\mathbf{v}}(x) \big]\ d V(x),
\end{equation}
where and also in what follows, for two matrices $\mathbf{A}=(a_{ij})_{i,j=1}^N$ and $\mathbf{B}=(b_{ij})_{i,j=1}^N$,
\[
\mathbf{A}:\mathbf{B}=\sum_{i,j=1}^N a_{ij}b_{ij}.
\]
For the solution $\mathbf{u}_\delta$ to \eqref{eq:lame1}, we define
\begin{equation}\label{eq:energy2}
\mathbf{E}_\delta(\mathbf{C}_{\widetilde{\lambda},\widetilde{\mu}}, \mathbf{f}):=\frac{\delta}{2}\mathbf{P}_{\lambda,\mu}(\mathbf{u}_\delta, \mathbf{u}_\delta).
\end{equation}
$\mathbf{E}_\delta$ defined above signifies the energy dissipated into heat of the elastostatic system \eqref{eq:lame1}.

\begin{defn}\label{def:resonant}
The configuration $(\mathbf{C}_{\widetilde{\lambda},\widetilde{\mu}}, \mathbf{f})$ in \eqref{eq:lame1} is said to be \emph{resonant} if
\begin{equation}\label{eq:def1}
\lim_{\delta\rightarrow+0}\mathbf{E}_\delta(\mathbf{C}_{\widetilde{\lambda},\widetilde{\mu}},\mathbf{f})=+\infty;
\end{equation}
and it is said to be \emph{weakly resonant} if
\begin{equation}\label{eq:def2}
\limsup_{\delta\rightarrow+0}\mathbf{E}_\delta(\mathbf{C}_{\widetilde{\lambda},\widetilde{\mu}},\mathbf{f})=+\infty.
\end{equation}
\end{defn}

As mentioned earlier, the resonance for the Lam\'e system \eqref{eq:lame1} has been investigated in \cite{AKKY} in two dimensions. The results that we shall establish in this paper distinguish from those in \cite{AKKY} in the following aspects.

\begin{enumerate}
\item The results in \cite{AKKY} were established by using the spectral approach, whereas in this paper we shall make use of variational approach to establish our resonance results.

\item In \cite{AKKY}, it is always assumed that the core is empty, namely $\Sigma=\emptyset$, whereas in this paper, $\Sigma$ could be an empty set, or a non-empty set of an arbitrary shape. In \cite{AKKY}, the exterior domain, namely $\Omega$ could be a disk or an ellipse, whereas in this paper, we take $\Omega$ to be a disk. Nevertheless, we would like to emphasize that our study can be extended to the case that $\Omega$ is an ellipse by following similar arguments.

\item In \cite{AKKY}, the force term $\mathbf{f}$ can be a function compactly supported in $\mathbb{R}^2\backslash\overline{\Omega}$, whereas in our study, we assume that $\mathbf{f}$ is distributed on a circular curve enclosing $\Omega$. Our method also allows for more general distributional sources by using the principle of superpositions.

\item If the core $\Sigma$ is an empty set, then in both \cite{AKKY} and the present paper, the resonance results are established. If the exterior domain $\Omega$ is an ellipse with $\Sigma=\emptyset$, then ALR and CALR results were established in \cite{AKKY}. In our study, by assuming that $\Omega$ is a disk and $\Sigma$ is nonempty, we establish the ALR result.

\end{enumerate}

We shall mainly focus on the two-dimensional case, namely $N=2$, and unless otherwise specified, we shall take that
\begin{equation}\label{eq:tensor3}
c:=-\frac{\lambda+\mu}{\lambda+3\mu}
\end{equation}
in \eqref{eq:tensor2}. Here, it is easily seen from \eqref{eq:convex} with $N=2$ that $c$ defined in \eqref{eq:tensor3} is negatively valued. Hence, \eqref{eq:tensor1} yields a plasmonic structure. In Section~\ref{sect:4d}, we shall remark our study for the three-dimensional case.

The rest of the paper is organized as follows. In Section 2, we establish the primal and dual variational principles for the elastostatic system. Section 3  is devoted to the resonance and non-resonance results. We conclude our study in Section 4 by discussing the extension to the three-dimensional case.

\section{Variational principles for the elastostatic system}

In this section, we establish the primal and variational principles for the elastostatic system \eqref{eq:lame1}, which shall play a critical role in our subsequent resonance study. Throughout the present section, we assume that the force term $\mathbf{f}=(\mathbf{f}_i)_{i=1}^N\in H^{-1}(\mathbb{R}^N)^N$ in \eqref{eq:lame1} with a compact support and a zero average in the sense that
\begin{equation}\label{eq:average}
\langle \mathbf{f}_i, \mathbf{1}\rangle=0,\quad i=1,2,\ldots,N,
\end{equation}
where $\mathbf{1}: \mathbb{R}^N\rightarrow\mathbb{R}$ is the constant function, $\mathbf{1}(x)=1$ for all $x\in\mathbb{R}^N$.

\subsection{Preliminaries for the elastostatic system}

We collect some preliminary knowledges on the elastostatic system \eqref{eq:lame1}. Those are standard results, but we cannot find a convenient reference. In what follows, we let $B_R$ with $R\in\mathbb{R}_+$ denote a central ball of radius $R$ in $\mathbb{R}^N$. Throughout, without loss of generality, we assume that there exists $R_0\in\mathbb{R}_+$ such that $supp(\mathbf{f})\subset B_{R_0}$. For a connected Lipschitz surface $\Gamma$ in $\mathbb{R}^N\backslash\overline{\Omega}$, enclosing a bounded domain $D$, we define the conormal derivative $\partial \bu_\delta/ \partial \bm{\nu}$ on the boundary $\Gamma=\partial D$ by
\begin{equation}\label{eq:normald}
\frac{\partial \bu_\delta}{\partial \bm{\nu}}= \widetilde\lambda(\nabla \cdot \ \bu_\delta)\bm{\nu}+ \widetilde\mu(\nabla \bu_\delta + \nabla \bu_\delta^T)\bm{\nu}\quad \mbox{on}\ \ \Gamma,
\end{equation}
where $\bm{\nu}$ is the outward unit normal vector to the boundary $\Gamma$. Let $\bu_\delta\in H_{\text{loc}}^1(\mathbb{R}^N)^N$ satisfy the Lam\'e system, namely, $\mathcal{L}_{\widetilde\lambda,\widetilde\mu}\bu_\delta=\mathbf{f}$. For any test function $\mathbf{v}\in H_{\text{loc}}^1(\mathbb{R}^N)^N$, we recall the following Green's formula,
\begin{equation}\label{eq:greenformula}
\int_{\partial D}\overline{\bv} \cdot\frac{\partial\bu_\delta}{\partial\bm{\nu}}\, ds(x)=\int_{D}\overline{\bv}\cdot\mathcal{L}_{\widetilde\lambda,\widetilde\mu}\bu_\delta\, dV(x)+\mathbf{P}_{\widetilde{\lambda},\widetilde{\mu}}(\bu_\delta,\bv),
\end{equation}
where $\mathbf{P}(\bu_\delta,\bv)$ is given by \eqref{eq:energy1} with the integration domain replaced by $D$. Using \eqref{eq:greenformula}, the weak solution $\bu_\delta\in H^1_{\text{loc}}(\mathbb{R}^N)^N$ to \eqref{eq:lame1} is given in the sense that
\begin{equation}\label{eq:weak1}
\mathbf{P}_{\widetilde\lambda,\widetilde\mu}(\bu_\delta,\bv)=\langle\mathbf{f},\overline\bv\rangle,
\end{equation}
where $\bv\in H_{\text{loc}}^1(\mathbb{R}^N)^N$, compactly supported in $B_R$ for any $R\geq R_0$. In order to have a functional analytic framework for the variational formulation, we introduce the following function space
\begin{equation}\label{eq:weak2}
\mathcal{S}:=\big\{ \bu\in H_{\text{loc}}^1(\mathbb{R}^N)^N;\ \nabla\bu\in L^2(\mathbb{R}^N)^{N\times N}\ \ \mbox{and}\ \ \int_{B_{R_0}}\bu=0 \big\}.
\end{equation}
Consider the following sesquilinear form, $\mathscr{B}(\cdot,\cdot): \mathcal{S}\times\mathcal{S}\rightarrow\mathbb{C}$, defined by
\begin{equation}\label{eq:bilinearform}
\mathscr{B}(\bu,\bv):=-\mathrm{i}\cdot\mathbf{P}_{\widetilde\lambda,\widetilde\mu}(\bu,\bv).
\end{equation}
It is straightforward to verify that $\mathscr{B}$ is bounded and by using the Poincar\'e-Wirtinger inequality, coercive. Hence, by the Lax-Milgram theorem, one can show that there exists a unique solution $\bu_\delta\in\mathcal{S}$ to the Lam\'e system \eqref{eq:lame1}. Finally, by using the fact that the Kelvin matrix of the fundamental solution  $\bm{\Phi}=(\Phi_{ij})_{i,j=1}^N$ to the PDO $\mathcal{L}_{\lambda,\mu}$ is given by (cf. \cite{Kup})
\begin{equation}\label{eq:funds}
\Phi_{ij}=\begin{cases}
\displaystyle{\frac{\alpha}{2\pi}\delta_{ij}\ln\|x\|-\frac{\beta}{2\pi}\frac{x_ix_j}{\|x\|^2}}\quad &\mbox{when}\ \  N=2,\medskip\\
\displaystyle{-\frac{\alpha}{4\pi}\frac{\delta_{ij}}{\|x\|}-\frac{\beta}{2\pi}\frac{x_ix_j}{\|x\|^3}}\quad &\mbox{when}\ \ N=3,
\end{cases}
\end{equation}
where
\[
\alpha=\frac 1 2\left(\frac 1 \mu+\frac{1}{2\mu+\lambda}\right)\quad\mbox{and}\quad \beta=\frac 1 2\left(\frac 1 \mu-\frac{1}{2\mu+\lambda}\right).
\]
In \eqref{eq:funds}, $x=(x_i)_{i=1}^N\in\mathbb{R}^N$ and $\delta_{ij}$ is the Kronecker delta. Using the form of the fundamental solution in \eqref{eq:funds}, one can show by direct calculations that the solution $\bu_\delta\in \mathcal{S}$ to $\mathcal{L}_{\widetilde\lambda,\widetilde\mu}\bu_\delta=\mathbf{f}$ possess the asymptotic behaviour, $\bu_\delta(x)=\mathcal{O}(\|x\|^{-1})$ as $\|x\|\rightarrow+\infty$.

\subsection{Primal and dual variational principles}

We now establish the primal and dual variational principles for the elastostatic system \eqref{eq:lame1}. For a fixed force term $\mathbf{f}\in H^{-1}(\mathbb{R}^N)^N$ and for the solution $\mathbf{u}_\delta\in H^1_{\text{loc}}(\mathbb{R}^N)^N: \mathbb{R}^N\rightarrow\mathbb{C}^N$, we set
\begin{equation}\label{eq:decomp1}
\mathbf{u}_\delta=\mathbf{v}_\delta+\mathrm{i}\frac{1}{\delta}\mathbf{w}_\delta,
\end{equation}
where $\mathbf{v}_\delta, \mathbf{w}_\delta\in H^1_{\text{loc}}(\mathbb{R}^N)^N:\mathbb{R}^N\rightarrow\mathbb{R}^N$ satisfying $\mathbf{v}_\delta=\mathcal{O}(\|x\|^{-1})$ and $\mathbf{w}_\delta=\mathcal{O}(\|x\|^{-1})$ as $\|x\|\rightarrow+\infty$. By straightforward calculations, one can show that the elastostatic system \eqref{eq:lame1} for $\mathbf{u}_\delta$ is equivalent to the following coupled system for the two real functions $\mathbf{v}_\delta$ and $\mathbf{w}_\delta$:
\begin{align}
& A\mathcal{L}_{\lambda,\mu}\mathbf{v}_\delta-\mathcal{L}_{\lambda,\mu}\mathbf{w}_\delta=\mathbf{f},\label{eq:couple1}\\
& A\mathcal{L}_{\lambda,\mu}\mathbf{w}_\delta+\delta^2\mathcal{L}_{\lambda,\mu}\mathbf{v}_\delta=0,\label{eq:couple2}
\end{align}
where $A$ is given in \eqref{eq:tensor2} and \eqref{eq:tensor3}, and $(\lambda, \mu)$ are the two regular Lam\'e constants in \eqref{eq:tensor1}. Furthermore, by direct computations, one has that
\begin{equation}\label{eq:energy3}
\begin{split}
\mathbf{E}_\delta(\mathbf{u}_\delta):=&\mathbf{E}_\delta(\mathbf{C}_{\widetilde{\lambda},\widetilde{\mu}},\mathbf{f})=\frac{\delta}{2}\mathbf{P}_{\lambda,\mu}(\mathbf{u}_\delta,\mathbf{u}_\delta)\\
=&\frac{\delta}{2}\mathbf{P}_{\lambda,\mu}(\mathbf{v}_\delta,\mathbf{v}_\delta)+\frac{1}{2\delta}\mathbf{P}_{\lambda,\mu}(\mathbf{w}_\delta,\mathbf{w}_\delta),
\end{split}
\end{equation}
where $\mathbf{E}_\delta$ is given in \eqref{eq:energy2} and $\mathbf{P}$ is given in \eqref{eq:energy1}.

Next, we introduce the following energy functional
\begin{equation}\label{eq:ef1}
\mathbf{I}_\delta(\mathbf{v},\mathbf{w}):=\frac\delta 2 \mathbf{P}_{\lambda,\mu}(\mathbf{v},\mathbf{v})+\frac{1}{2\delta}\mathbf{P}_{\lambda,\mu}(\mathbf{w},\mathbf{w})\quad\mbox{for}\ \ \ (\mathbf{v}, \mathbf{w})\in \mathcal{S}\times\mathcal{S},
\end{equation}
and consider the following optimization problem:
\begin{equation}\label{eq:primal}
\begin{split}
&\mbox{Minimize $\mathbf{I}_\delta(\mathbf{v},\mathbf{w})$ over all pairs $(\mathbf{v},\mathbf{w})\in \mathcal{S}\times\mathcal{S}$ }\\
&\mbox{subject to the PDE constraint } A\mathcal{L}_{\lambda,\mu}\mathbf{v}-\mathcal{L}_{\lambda,\mu}\mathbf{w}=\mathbf{f}.
\end{split}
\end{equation}
In the sequel, we shall refer to \eqref{eq:primal} as the primal variational problem for the elastostatic system \eqref{eq:lame1}, or equivalently \eqref{eq:couple1}-\eqref{eq:couple2}.

We have the following result about the primal problem \eqref{eq:primal}.

\begin{lem}\label{lem:primal}
The primal variational problem \eqref{eq:primal} is equivalent to the elastic problem \eqref{eq:lame1} in the following sense.
\begin{enumerate}
\item The infimum
\begin{equation}\label{eq:primalp1}
\inf \big\{\mathbf{I}_\delta(\mathbf{v},\mathbf{w});\, A\mathcal{L}_{\lambda,\mu}\mathbf{v}-\mathcal{L}_{\lambda,\mu}\mathbf{w}=\mathbf{f} \big\}
\end{equation}
is attainable at a pair $(\mathbf{v}_\delta,\mathbf{w}_\delta)\in \mathcal{S}\times\mathcal{S}$.\medskip

\item The minimizing pair $(\mathbf{v}_\delta,\mathbf{w}_\delta)$ is unique, up to an additive constant, such that the function $\mathbf{u}_\delta:=\mathbf{v}_\delta+\mathrm{i}\delta^{-1}\mathbf{w}_\delta$ is the unique solution to the elastic problem \eqref{eq:lame1}.\medskip

\item For the solution in (2), the energies coincide, namely,
\begin{equation}\label{eq:primalp2}
\mathbf{E}_\delta(\mathbf{u}_\delta)=\mathbf{I}_\delta(\mathbf{v}_\delta,\mathbf{w}_\delta).
\end{equation}
\end{enumerate}
\end{lem}

\begin{proof}

It is directly verified that $\mathbf{I}_\delta(\mathbf{v}_\delta,\mathbf{w}_\delta)$ is a convex functional and the PDE constraint in \eqref{eq:primal} yields a nonempty set. Hence, the infimum of \eqref{eq:primalp1} is attainable; that is, there exists a pair $(\mathbf{v}_\delta,\mathbf{w}_\delta)$ such that
\begin{equation}\label{eq:primalp3}
\mathbf{I}_\delta(\mathbf{v}_\delta,\mathbf{w}_\delta)\leq \mathbf{I}_\delta(\mathbf{v},\mathbf{w})\quad\mbox{for all}\ (\mathbf{v},\mathbf{w})\ \ \mbox{satisfying}\ \ A\mathcal{L}_{\lambda,\mu}\mathbf{v}-\mathcal{L}_{\lambda,\mu}\mathbf{w}=\mathbf{f}.
\end{equation}
Next we show that $\mathbf{u}_\delta:=\mathbf{v}_\delta+\mathrm{i}\delta^{-1}\mathbf{w}_\delta$ is the solution to the elastic problem \eqref{eq:lame1}.

As a minimizer of $\mathbf{I}_\delta$, the pair $(\mathbf{v}_\delta,\mathbf{w}_\delta)$ must verify the Euler-Lagrange equation,
\begin{equation}\label{eq:primalp4}
\partial_\tau\mathbf{I}_\delta(\mathbf{v}_\delta+\tau\widetilde{\mathbf{v}}, \mathbf{w}_\delta+\tau\widetilde{\mathbf{w}})\big|_{\tau=0}=0,
\end{equation}
for every pair $(\widetilde{\mathbf{v}},\widetilde{\mathbf{w}})$ satisfying
\begin{equation}\label{eq:primalp5}
A\mathcal{L}_{\lambda,\mu}\widetilde{\mathbf{v}}-\mathcal{L}_{\lambda,\mu}\widetilde{\mathbf{w}}=0.
\end{equation}
For the energy $\mathbf{I}_{\delta}$, this equation reads
\begin{equation}
\delta \mathbf{P}_{\lambda,\mu}(\bv_{\delta}, \tilde{\bv})+\frac{1}{\delta} \mathbf{P}_{\lambda,\mu}(\bw_{\delta},\tilde{\bw})=0.
\end{equation}
With the help of Green's formula and (\ref{eq:primalp5}), the last equation yields
\begin{equation}
\begin{split}
  & -\delta \int_{\mathbb{R}^N} \cl_{\lambda,\mu} \bv_{\delta} \cdot \tilde{\bv}dV(x) -\frac{1}{\delta} \int_{\mathbb{R}^N} \bw_{\delta} \cdot  \cl_{\lambda,\mu} \tilde{\bw}d V(x) \\
  & = -\delta \int_{\mathbb{R}^N} \cl_{\lambda,\mu} \bv_{\delta} \cdot \tilde{\bv}dV(x) -\frac{1}{\delta} \int_{\mathbb{R}^N} \bw_{\delta} \cdot A \cl_{\lambda,\mu} \tilde{\bv}d V(x) \\
  & = -\frac{1}{\delta}\int_{\mathbb{R}^N}(\delta^2 \cl_{\lambda,\mu} \bv_{\delta} +A \cl_{\lambda,\mu}\bw_{\delta}) \cdot \tilde{\bv} dV(x) =0
\end{split}
\end{equation}
which is the weak form of (\ref{eq:couple2}). As a solution of (\ref{eq:couple1})-(\ref{eq:couple2}), the pair $(\bv_{\delta},\bw_{\delta})$ defines through $\bu_{\delta}:=\bv_{\delta} + \mathrm{i}\delta^{-1}\bw_{\delta}$ a solution to the original problem (\ref{eq:lame1}).

The uniqueness is a consequence of the fact that the original problem (\ref{eq:lame1}) possesses a unique solution. Finally, it is obvious that $\mathbf{E}_{\delta}(\bu_{\delta})=\mathbf{I}_{\delta}(\bv_{\delta},\bw_{\delta})$ from (\ref{eq:energy3}).

The proof is complete.
\end{proof}

\begin{rem}\label{rem:primal}
By Lemma~\ref{lem:primal}, one can readily see that for the solution $\mathbf{u}_\delta$ to \eqref{eq:lame1} and the energy $\mathbf{E}_\delta(\mathbf{u}_\delta)$ in \eqref{eq:energy3}, one has that
\begin{equation}\label{eq:primalp6}
\mathbf{E}_\delta(\mathbf{u}_\delta)\leq \mathbf{I}_\delta(\mathbf{v},\mathbf{w}),
\end{equation}
for the energy functional $\mathbf{I}_\delta$ defined in \eqref{eq:ef1} and every pair $(\mathbf{v},\mathbf{w})$ verifying the constraint specified in \eqref{eq:primal}. For the elastic configuration $(\mathbf{C}_{\widetilde{\lambda},\widetilde{\mu}},\mathbf{f})$, we shall make use of the primal variational principle via \eqref{eq:primalp6} to show the non-resonance result by constructing suitable test functions $\mathbf{v}$ and $\mathbf{w}$.
\end{rem}

We proceed to introduce the dual variational problem by defining the following energy functional
\begin{equation}\label{eq:ef2}
\mathbf{J}_\delta(\mathbf{v},\bm{\psi}):=\int_{\mathbb{R}^N}\mathbf{f}\cdot \bm{\psi}-\frac{\delta}{2}\mathbf{P}_{\lambda,\mu}(\mathbf{v},\mathbf{v})-\frac{\delta}{2}\mathbf{P}_{\lambda,\mu}(\bm{\psi},\bm{\psi})\ \mbox{for} \ (\mathbf{v}, \bm{\psi})\in \mathcal{S}\times\mathcal{S}.
\end{equation}
Consider the following optimization problem
\begin{equation}\label{eq:dual}
\begin{split}
&\mbox{Maximize $\mathbf{J}_\delta(\mathbf{v},\bm{\psi})$ over all pairs $(\mathbf{v},\bm{\psi})\in \mathcal{S}\times\mathcal{S}$}\\
&\mbox{subject to the PDE constraint } A\mathcal{L}_{\lambda,\mu}\bm{\psi}+\delta\mathcal{L}_{\lambda,\mu}\mathbf{v}=0.
\end{split}
\end{equation}
In the sequel, we shall refer to \eqref{eq:dual} as the dual variational problem for the elastostatic system \eqref{eq:lame1}, or equivalently \eqref{eq:couple1}-\eqref{eq:couple2}.

We have the following result about the dual problem \eqref{eq:dual}.

\begin{lem}\label{lem:dual}
The dual variational problem \eqref{eq:dual} is equivalent to the elastic problem \eqref{eq:lame1} in the following sense.
\begin{enumerate}
\item The supremum
\begin{equation}\label{eq:dualp1}
\sup \big\{\mathbf{J}_\delta(\mathbf{v},\bm{\psi}); A\mathcal{L}_{\lambda,\mu}\bm{\psi}+\delta\mathcal{L}_{\lambda,\mu}\bv=0 \big\}
\end{equation}
is attainable at a pair $(\mathbf{v}_\delta,\bm{\psi}_\delta)\in \mathcal{S}\times\mathcal{S}$.\medskip

\item The maximizing pair $(\mathbf{v}_\delta,\bm{\psi}_\delta)$ is unique, up to an additive constant, such that the function $\mathbf{u}_\delta:=\mathbf{v}_\delta+\mathrm{i}\bm{\psi}_\delta$ is the unique solution to the elastic problem \eqref{eq:lame1}.\medskip

\item For the solution in (2), the energies coincide, namely,
\begin{equation}\label{eq:dualp2}
\mathbf{E}_\delta(\mathbf{u}_\delta)=\mathbf{J}_\delta(\mathbf{v}_\delta,\bm{\psi}_\delta).
\end{equation}
\end{enumerate}
\end{lem}

\begin{proof}
We shall follow similar arguments to those for the proof of Lemma~\ref{lem:primal}. The existence and uniqueness can be proved similarly to that of the primal variational problem.

Next we prove that $\bu_{\delta}:=\bv_{\delta} + \mathrm{i}\bm{\psi}_{\delta}$ is the solution of the original problem (\ref{eq:lame1}). As a maximizer of $J_{\delta}$, the pair $(\bv_{\delta},\bm{\psi}_{\delta})$ must verify the Euler-lagrange equation,
\begin{equation}\label{eq:el2}
\partial_{\tau} \mathbf{J}_{\delta}(\bv_{\delta}+\tau \tilde{\bv} ,\bm{\psi}_{\delta}+\tau \tilde{\bm{\psi}})|_{\tau=0}=0
\end{equation}
for every pair $(\tilde{\bv},\tilde{\bm{\psi}})$ satisfying
\begin{equation}\label{eq:dualp3}
A\cl_{\lambda,\mu}\tilde{\bm{\psi}} + \delta \cl_{\lambda,\mu}\tilde{\bv}=0.
\end{equation}
For the energy functional $\mathbf{J}_{\delta}$, this equation is equivalent to
\begin{equation}\label{eq:el3}
\int_{\mathbb{R}^N} \bff \cdot \tilde{\bm{\psi}} -\delta \mathbf{P}_{\lambda,\mu}(\bv_{\delta},\tilde{\bv})-\delta \mathbf{P}_{\lambda,\mu}(\bm{\psi}_{\delta},\tilde{\bm{\psi}})=0.
\end{equation}
Using Green's formula and (\ref{eq:dualp3}), together with straightforward calculations, one has from \eqref{eq:el3} that
\begin{equation}\label{eq:el4}
\begin{split}
  & \int_{\mathbb{R}^N} \bff \cdot \tilde{\bm{\psi}}\ dV(x) + \delta \int_{\mathbb{R}^N}  \bv_{\delta} \cdot \cl_{\lambda,\mu} \tilde{\bv}\ dV(x) + \delta \int_{\mathbb{R}^N} \bm{\psi}_{\delta} \cdot  \cl_{\lambda,\mu} \tilde{\bm{\psi}}\ dV(x) \\
 = & \int_{\mathbb{R}^N} \bff \cdot \tilde{\bm{\psi}}\ dV(x) - \int_{\mathbb{R}^N} \bv_{\delta} \cdot A\cl_{\lambda,\mu}\tilde{\bm{\psi}}\ dV(x) + \delta \int_{\mathbb{R}^N} \cl_{\lambda,\mu} \bm{\psi}_{\delta} \cdot  \tilde{\bm{\psi}}\ dV(x) \\
 = &  \int_{\mathbb{R}^N} \bff \cdot \tilde{\bm{\psi}}\ dV(x)- \int_{\mathbb{R}^N} A\cl_{\lambda,\mu} \bv_{\delta} \cdot \tilde{\bm{\psi}}\ dV(x) + \delta \int_{\mathbb{R}^N} \cl_{\lambda,\mu} \bm{\psi}_{\delta} \cdot  \tilde{\bm{\psi}}\ dV(x) \\
 = &  \int_{\mathbb{R}^N}(\bff - A\cl_{\lambda,\mu} \bv_{\delta} + \delta \cl_{\lambda,\mu} \bm{\psi}_{\delta} ) \cdot  \tilde{\bm{\psi}}\ dV(x) = 0.
\end{split}
\end{equation}
By \eqref{eq:el4}, we conclude that the pair $(\bv_{\delta},\bw_{\delta}):= (\bv_{\delta}, \delta \bm{\psi}_{\delta})$ is a weak solution of (\ref{eq:couple1})-(\ref{eq:couple2}), and thus $\bu_{\delta}:=\bv_{\delta} + i\bm{\psi}_{\delta}$ is a solution to the original problem (\ref{eq:lame1}).

Finally, we verify that $\mathbf{E}_{\delta}(\bu_{\delta})=\mathbf{J}_{\delta}(\bv_{\delta},\bm{\psi}_{\delta})$. By using Green's formula again, we have
\begin{equation}
\begin{split}
  & \mathbf{E}_{\delta}(\bu_{\delta}) - \mathbf{J}_{\delta}(\bv_{\delta},\bm{\psi}_{\delta})\\
  =  &  \frac{\delta}{2}\mathbf{P}_{\lambda,\mu}(\bv_{\delta},\bv_{\delta})+\frac{\delta}{2}\mathbf{P}_{\lambda,\mu}(\bm{\psi}_{\delta},\bm{\psi}_{\delta}) - \int_{\mathbb{R}^N} \bff \cdot \bm{\psi}_{\delta} + \frac{\delta}{2}\mathbf{P}_{\lambda,\mu}(\bv_{\delta},\bv_{\delta})+\frac{\delta}{2}\mathbf{P}_{\lambda,\mu}(\bm{\psi}_{\delta},\bm{\psi}_{\delta}) \\
  =  &  \delta \mathbf{P}_{\lambda,\mu}(\bv_{\delta},\bv_{\delta})+\delta \mathbf{P}_{\lambda,\mu}(\bm{\psi}_{\delta},\bm{\psi}_{\delta}) - \int_{\mathbb{R}^N} \bff \cdot \bm{\psi}_{\delta} \\
  =  &  -\delta \int_{\mathbb{R}^N} \bv_{\delta} \cdot \cl_{\lambda,\mu} \bv_{\delta} -\delta \int_{\mathbb{R}^N} \bm{\psi}_{\delta} \cdot \cl_{\lambda,\mu} \bm{\psi}_{\delta} - \int_{\mathbb{R}^N} \bff \cdot \bm{\psi}_{\delta}\\
  =  &  \int_{\mathbb{R}^N} \bv_{\delta} \cdot A\cl_{\lambda,\mu} \bm{\psi}_{\delta} -\delta \int_{\mathbb{R}^N} \cl_{\lambda,\mu} \bm{\psi}_{\delta} \cdot  \bm{\psi}_{\delta} - \int_{\mathbb{R}^N} \bff \cdot \bm{\psi}_{\delta}\\
  =  &   \int_{\mathbb{R}^N}( A\cl_{\lambda,\mu} \bv_{\delta} - \delta \cl_{\lambda,\mu} \bm{\psi}_{\delta} -\bff ) \cdot \bm{\psi}_{\delta}\ dV(x) =0.
\end{split}
\end{equation}

The proof is complete.
\end{proof}

\begin{rem}\label{rem:dual}
By Lemma~\ref{lem:dual}, one can readily see that for the solution $\mathbf{u}_\delta$ to \eqref{eq:lame1} and the energy $\mathbf{E}_\delta(\mathbf{u}_\delta)$ in \eqref{eq:energy3}, one has that
\begin{equation}\label{eq:dualp4}
\mathbf{E}_\delta(\mathbf{u}_\delta)\geq \mathbf{J}_\delta(\mathbf{v},\bm{\psi}),
\end{equation}
for the energy functional $\mathbf{J}_\delta$ defined in \eqref{eq:ef2} and every pair $(\bv,\bm{\psi})$ verifying the constraint specified in \eqref{eq:dual}. For the elastic configuration $(\mathbf{C}_{\widetilde{\lambda},\widetilde{\mu}},\mathbf{f})$, we shall make use of the dual variational principle via \eqref{eq:dualp4} to show the resonance result by constructing suitable test functions $\mathbf{v}$ and $\bm{\psi}$.
\end{rem}

\section{Anomalous localized resonance for the elastostatic system}

In this section, we are in a position to present the ALR results for the Lam\'e system \eqref{eq:lame1} with the elastic configuration $(\mathbf{C}_{\widetilde\lambda,\widetilde\mu}, \mathbf{f})$ in two dimensions. Henceforth, we assume that the force term $f(x)$ is of the following form
\begin{equation}\label{eq:force1}
\mathbf{f}=\mathbf{F}\mathcal{H}^1 \lfloor \partial B_q,\quad \mathbf{F}: \partial B_q\rightarrow \mathbb{R}^2,\quad \mathbf{F}\in L^2(\partial B_q)^2,\ \ q\in\mathbb{R}_+,
\end{equation}
and
\begin{equation}\label{eq:force2}
\int_{\partial B_q} \mathbf{F}\ d \mathcal{H}^1=0.
\end{equation}
Moreover, we let the exterior domain $\Omega$ for the plasmonic structure \eqref{eq:tensor2} be taken to be $B_R$ with a fixed $R\in\mathbb{R}_+$.

\subsection{Perfect plasmone elastic waves}

As discussed earlier in Remarks \ref{rem:primal} and \ref{rem:dual}, in order to show the resonance and non-resonance results by using the variational principles, we shall construct suitable trial functions. Those functions are referred to as {\it perfect plasmone elastic waves}, and are contained in the following lemma. Starting from now on and throughout the rest of the paper, we make use of the polar coordinates $x=(r\cos\theta, r\sin\theta)\in\mathbb{R}^2$.

\begin{lem}\label{lem:perfectwaves}
Consider the PDE for a function $\bm{\psi}: \mathbb{R}^2 \rightarrow \mathbb{R}^2$
\begin{equation}\label{eq:trial1}
\begin{split}
  & A\cl_{\lambda,\mu}\bm{\psi}=0, \\
  & \bm{\psi}(x)= \mathcal{O}(\|x\|^{-1}) \quad \mbox{as} \quad \|x\| \rightarrow \infty,
\end{split}
\end{equation}
where
\begin{equation}\label{eq:A2}
   A(x) = \begin{cases}
   c,\quad & |x| \leq R,\\
   +1,\quad & |x| >R,
   \end{cases}
\end{equation}
and $c$ is given in (\ref{eq:tensor3}). Then there exist non-trivial solutions $\bm{\psi} = \widehat{\bm{\psi}}_k$, $k=1,2,\ldots$, which are given as follows:
\begin{equation}\label{eq:solution1}
\widehat{\bm{\psi}}_k(x):=\begin{cases}
& R^{2k} \left[
  \begin{array}{c}
   r^{-k} \cos(k \theta) + k \alpha(r^2-R^2) \frac{1} {r^{k+2}} \cos((k+2) \theta)  \\
   -r^{-k} \sin(k \theta) + k \alpha(r^2-R^2) \frac{1} {r^{k+2}} \sin((k+2) \theta)   \\
  \end{array}
\right], r>R;\\
& \left[
  \begin{array}{c}
     r^k \cos(k \theta)   \\
    -r^k \sin(k \theta)  \\
  \end{array}
\right],
\quad \quad r<R ;
\end{cases}
\end{equation}
where
\begin{equation}\label{coeff_alpha}
  \alpha = \alpha_2 / \alpha_1,
\end{equation}
 and
\begin{equation}\label{coeff_alpha_1}
  \alpha_1=\frac{1}{2}\left( \frac{1}{\mu} + \frac{1}{2 \mu + \lambda}  \right),
\end{equation}
\begin{equation}\label{coeff_alpha_2}
  \alpha_2=\frac{1}{2}\left( \frac{1}{\mu} - \frac{1}{2 \mu + \lambda}  \right).
\end{equation}
Moreover, by straightforward calculations, one has that that the energy $P(\widehat{\bm{\psi}}_k,\widehat{\bm{\psi}}_k)$ is
\begin{equation}
 \mathbf{P}_{\lambda,\mu}(\widehat{\bm{\psi}}_k,\widehat{\bm{\psi}}_k)= 8k\pi \frac{\mu(\lambda + 2 \mu)}{\lambda + 3 \mu} R^{2k} .
\end{equation}
\end{lem}

\begin{proof}
The lemma can be verified by straightforward computations.
\end{proof}

\begin{rem}
For the subsequent use, we remark that if
\begin{equation}\label{eq:A3}
   A(x) = \begin{cases}
   +1,\quad & |x| \leq R,\\
   c,\quad & |x| >R ,
   \end{cases}
\end{equation}
then by straightforward calculations, one can verify that the perfect plasmone elastic waves $\widehat{\bm{\psi}}_k(x)$, $k=2,3,\ldots,$ are given by:
\begin{equation}\label{eq:solution56}
\widehat{\bm{\psi}}_k(x):=\begin{cases}
\left[
  \begin{array}{c}
    R^{2k} r^{-k} \cos(k \theta) \\
   R^{2k} r^{-k} \sin(k \theta) \\
  \end{array}
\right]
\quad & r>R;\\
\left[
  \begin{array}{c}
     r^k \cos(k \theta) - k \alpha (r^2-R^2)  r^{k-2} \cos((k-2) \theta)  \\
     r^k \sin(k \theta) + k \alpha (r^2-R^2)  r^{k-2} \sin((k-2) \theta) \\
  \end{array}
\right]
\quad & r<R.
\end{cases}
\end{equation}

\end{rem}

\smallskip

Next we give the Fourier series expression of the force term $\bff$. Suppose that the source $\bff$ is real-valued and supported at distance $q$ from origin, given in \eqref{eq:force1} and \eqref{eq:force2}, then it can be represented as follows
\begin{equation}\label{eq:source22}
  \bff = \sum_{k=1}^{\infty} (\beta_k \bff_{1,k}^q + \gamma_k \bff_{2,k}^q+\xi_k \bff_{3,k}^q+\eta_k\bff_{4,k}^q)
\end{equation}
where
\begin{align}
\bff_{1,k}^q =&
\left[
  \begin{array}{c}
     \cos(k \theta)  \\
     \sin(k \theta)  \\
  \end{array}
\right]
 \mathcal{H}^1 \lfloor \partial B_q ,\label{eq:fkq1}\\
\bff_{2,k}^q =&
\left[
  \begin{array}{c}
     \cos(k \theta)  \\
     -\sin(k \theta)  \\
  \end{array}
\right]
 \mathcal{H}^1 \lfloor \partial B_q,\label{eq:fkq2}\\
  \bff_{3,k}^q =&
\left[
  \begin{array}{c}
     -\sin(k \theta)  \\
     \cos(k \theta)  \\
  \end{array}
\right]
 \mathcal{H}^1 \lfloor \partial B_q ,\label{eq:fkq3}\\
\bff_{4,k}^q =&
\left[
  \begin{array}{c}
     \sin(k \theta)  \\
     \cos(k \theta)  \\
  \end{array}
\right]
 \mathcal{H}^1 \lfloor \partial B_q,\label{eq:fkq4}
\end{align}
and
\begin{align}
& \beta_k= \int_{\partial B_q} \bff \cdot \bff_{1,k}^q\ ds, \quad  \gamma_k = \int_{\partial B_q} \bff \cdot \bff_{2,k}^q\ ds,\\
 & \xi_k= \int_{\partial B_q} \bff \cdot \bff_{3,k}^q\ ds, \quad  \eta_k = \int_{\partial B_q} \bff \cdot \bff_{4,k}^q\ ds.
\end{align}

In order to simplify the exposition, in our subsequent study, we shall always assume that $\xi_k=\eta_k\equiv 0$. That is, we exclude the presence of the modes $\bff_{3,k}$ and $\bff_{4,k}$ in the force term, and hence instead of the general form \eqref{eq:source22}, we shall consider a force term of the following form
\begin{equation}\label{eq:source2}
  \bff = \sum_{k=1}^{\infty} (\beta_k \bff_{1,k}^q + \gamma_k \bff_{2,k}^q).
\end{equation}
However, it is emphasized that all of the resonance and non-resonance results in the present paper still hold with the presence of the modes $\bff_{3,k}$ and $\bff_{4,k}$, by following completely similar arguments with necessary modifications.

\subsection{Resonance with no core}

We first consider the case that there is no core, namely $\Sigma=\emptyset$, in the plasmonic structure \eqref{eq:tensor1}--\eqref{eq:tensor2}. In this case, we have that the elastic configuration $(\mathbf{C}_{\widetilde\lambda,\widetilde\mu},\bff)$ is always resonant in the sense that

\begin{thm}\label{thm:nocore}
Consider the elastic configuration $(\mathbf{C}_{\widetilde\lambda,\widetilde\mu},\bff)$, where $\mathbf{C}_{\widetilde\lambda,\widetilde\mu}$ is described in \eqref{eq:tensor1}--\eqref{eq:tensor2} with $c$ given in \eqref{eq:tensor3} and $\Omega=B_R$ for a certain $R\in\mathbb{R}_+$. Let $\bff$ be given by (\ref{eq:source2}), with $\gamma_k \neq 0$ for some $k\in\mathbb{N}$, representing the force supported at a distance $q>R$. Assume that there is no core; that is, $\Sigma=\emptyset$. Then the configuration is resonant, i.e. $\mathbf{E}_{\delta}(\mathbf{C}_{\widetilde\lambda,\widetilde\mu},\bff) \rightarrow +\infty$ as $\delta \rightarrow +0$.
\end{thm}

\begin{proof}
We shall make use the dual variational principle for its proof. Fix the radii $R, q$ and consider an arbitrary sequence $\delta=\delta_j \rightarrow +0$ as $j\rightarrow+\infty$. Our aim is to find a sequence $(\bv_{\delta},\bm{\psi}_{\delta})$, satisfying the constraint $A\cl_{\lambda,\mu} \bm{\psi}_{\delta} + \delta \cl_{\lambda,\mu}\bv_{\delta}=0$ of \eqref{eq:dual} and such that $\mathbf{J}_{\delta}(\bv_{\delta},\bm{\psi}_{\delta})\rightarrow +\infty $. We choose
\begin{equation}
  \bv_{\delta} \equiv 0
\end{equation}
\begin{equation}
  \bm{\psi}_{\delta} : \equiv \tau_{\delta} \widehat{\bm{\psi}}_k,
\end{equation}
where $\widehat{\bm{\psi}}_k$ is given by (\ref{eq:solution1})--\eqref{coeff_alpha_2} and $\tau_{\delta} \in \mathbb{R}$ is to be chosen below. Thus the pair $(\bv_{\delta},\bm{\psi}_{\delta})$ satisfies the PDE constraint in (\ref{eq:dual}). With the help of (\ref{eq:dualp4}), the definition of $\mathbf{J}_{\delta}$, the orthogonality of Fourier series and $\gamma_k \neq 0$ for some $k\in\mathbb{N}$, we can obtain
\begin{equation}
  \begin{split}
    \mathbf{E}_{\delta}(\bu_{\delta}) & \geq  \mathbf{J}_{\delta}(\bv_{\delta},\bm{\psi}_{\delta}) = \mathbf{J}_{\delta}( 0 ,\bm{\psi}_{\delta}) =\int \bff \cdot \bm{\psi}_{\delta} -\frac{\delta}{2}\mathbf{P}_{\lambda,\mu}(\bm{\psi}_{\delta},\bm{\psi}_{\delta}) \\
      & = \int_{\partial B_q} \gamma_k \tau_{\delta} q^{-k} R^{2k}(\cos^2 (k \theta) + \sin^2 (k \theta)) -\frac{\delta}{2} |\tau_{\delta}|^2 \mathbf{P}_{\lambda,\mu}(\widehat{\bm{\psi}}_k,\widehat{\bm{\psi}}_k) \\
      & = 2\pi q  \gamma_k \tau_{\delta} q^{-k} R^{2k} - (\delta |\tau_{\delta}|^2 ) 4k\pi \frac{\mu(\lambda + 2 \mu)}{\lambda + 3 \mu} R^{2k}.
  \end{split}
\end{equation}
Choosing $\tau_{\delta} \rightarrow +\infty$ with $\delta |\tau_{\delta}|^2 \rightarrow +0$ as $\delta\rightarrow+0$, we obtain $ \mathbf{E}_{\delta}(\bu_{\delta})\rightarrow +\infty  $ for $\delta \rightarrow +0$.

The proof is complete.
\end{proof}

\begin{rem}\label{rem:modification1}
In Theorem~\ref{thm:nocore}, we assume $\gamma_k\neq 0$ for some $k\in\mathbb{N}$; that is, there is at least a mode $\bff_{2,k}$ presented in the force term $\bff$. If we assume that in the force term \eqref{eq:source2}, there is a coefficient $\beta_k \neq 0$ for some $k \geq 2$, then by following a completely similar argument, together with the modification of the plasmone constant $c$ in (\ref{eq:tensor2}) to be
  \begin{equation}\label{eq:mc1}
  c:= -\frac{\lambda+3\mu}{\lambda+\mu},
  \end{equation}
 one can draw a similar conclusion to Theorem~\ref{thm:nocore} about the resonance.
\end{rem}

\subsection{ALR with a core of an arbitrary shape}\label{sect:ALR1}

In this section, we consider a non-radial geometry with a core, $\Sigma \subset B_1$, of an arbitrary shape. We shall show that anomalous localized resonance occurs; that is, the resonance of the configuration depends on the location of the force term $\bff$.

\begin{thm}\label{thm:main2}
Consider the elastic configuration $(\mathbf{C}_{\widetilde\lambda,\widetilde\mu},\bff)$, where $\mathbf{C}_{\widetilde\lambda,\widetilde\mu}$ is described in \eqref{eq:tensor1}--\eqref{eq:tensor2} with $c$ given in \eqref{eq:tensor3} and, $\Omega=B_R$ for a certain $R>1$ and $\Sigma\subset B_1$ with a connected Lipschitz boundary $\partial \Sigma$.
 Consider the elastostatic system (\ref{eq:lame1}), with $\mathbf{C}_{\widetilde\lambda,\widetilde\mu}$ described above.  Then for every radius $R < q < R^*:=R^{3/2}$, there exists a source $\bff$ of the form \eqref{eq:source2} supported at a distance $q$ from the origin, such that the configuration $(\mathbf{C}_{\widetilde\lambda,\widetilde\mu},\bff)$ is resonant.
\end{thm}
\begin{proof}
 We fix $R < q < R^*$ and a sequence $\delta=\delta_j \rightarrow +0$ and consider a force term $\bff$ given by (\ref{eq:source2}). Our aim is to find a sequence $(\bv_{\delta},\bm{\psi}_{\delta})$, satisfying the PDE constraint $A\cl_{\lambda,\mu} \bm{\psi}_{\delta} + \delta \cl_{\lambda,\mu}\bv_{\delta}=0$ in (\ref{eq:dual}) and such that $\mathbf{J}_{\delta}(\bv_{\delta},\bm{\psi}_{\delta})\rightarrow +\infty $.

We choose
\begin{equation}
  \bm{\psi}_{\delta} : \equiv \tau_{\delta} \widehat{\bm{\psi}}_{k_{\delta}},
\end{equation}
where $\widehat{\bm{\psi}}_k$ is given by (\ref{eq:solution1})-(\ref{coeff_alpha_2}). The number $k=k_{\delta} \in \mathbb{N}$ and $\tau_{\delta} \in \mathbb{R}$ will be properly chosen below. For $\bm{\psi}_{\delta}$, it is apparent that $A\cl_{\lambda,\mu} \bm{\psi}_{\delta} \neq0$ along the core interface $\partial \Sigma \subset B_1$. In order to satisfy the PDE constraint we define $\bv_{\delta}$ to be the solution of $ -\delta \cl_{\lambda,\mu}\bv_{\delta} = A\cl_{\lambda,\mu} \bm{\psi}_{\delta}$. Since $- \cl_{\lambda,\mu}$ is an elliptic PDO, by the standard elliptic estimates one can arrive at the following estimate:
\begin{equation}
  \delta \mathbf{P}_{\lambda,\mu}(\bv_{\delta}, \bv_{\delta}) \leq C_{\lambda,\mu} \delta^{-1} \| A\cl_{\lambda,\mu} \bm{\psi}_{\delta} \|^2_{H^{-1}(\mathbb{R}^2)^2} \leq C_{\lambda,\mu} \delta^{-1} \tau_{\delta}^2 k_{\delta},
\end{equation}
where and also in what follows $C_{\lambda,\mu}$ denotes a generic positive constant. Indeed, one has by direct calculations that
\begin{equation}
 \mathbf{P}_{\lambda,\mu}(\bv_{\delta}, \bv_{\delta}) = \int_{\mathbb{R}^2}(\lambda |\nabla \cdot \bv_{\delta}|^2 + 2\mu  |\nabla^s \bv_{\delta}|^2 ) \leq C_{\lambda,\mu} \|  \bv_{\delta} \|^2_{H^{1}(\mathbb{R}^2)^2}.
\end{equation}
Since $ A\cl_{\lambda,\mu} \bm{\psi}_{\delta}=0$ outside $B_1$, we choose $\bmm{\omega} \in H_0^1(B_1)^2$ with $ \|\bmm{\omega}\|_{H_0^1(B_1)^2}=1 $ and we then have
\begin{equation}
\begin{split}
    & \quad \int_{\mathbb{R}^2}  \left(A\cl_{\lambda,\mu} \bm{\psi}_{\delta} \right) \cdot \bmm{\omega}  =  \int_{B_1}  \left(A\cl_{\lambda,\mu} \bm{\psi}_{\delta} \right) \cdot \bmm{\omega} \\
    & = A(\lambda +\mu)\int_{B_1} (\nabla \cdot \bm{\psi}_{\delta})(\nabla \cdot \bmm{\omega}) + A \mu \int_{B_1} \left( (\nabla \psi_{\delta}^1) \cdot (\nabla \omega_1) + (\nabla \psi_{\delta}^2) \cdot (\nabla \omega_2) \right) \\
    & \leq  A(\lambda +\mu) \| \nabla \cdot \bm{\psi}_{\delta} \|_{L^2(B_1)} + A \mu \left( \| \nabla \psi_{\delta}^1 \|_{L^2(B_1)^2} + \| \nabla \psi_{\delta}^2 \|_{L^2(B_1)^2} \right)\\
    & \leq C_{\lambda,\mu} \tau_{\delta} k_{\delta}^{1/2},
\end{split}
\end{equation}
where
\begin{equation}
  \bm{\psi}_{\delta} =
 \left[
   \begin{array}{c}
     \psi_{\delta}^1 \\
     \psi_{\delta}^2 \\
   \end{array}
 \right]
\quad \text{and} \quad
\bmm{\omega} =
 \left[
   \begin{array}{c}
     \omega_1 \\
     \omega_2 \\
   \end{array}
 \right]
\end{equation}
It remains to calculate the energy $\mathbf{J}_{\delta}(\bv_{\delta},\bm{\psi}_{\delta})$. {\color{black} We choose $k_{\delta}$ to be the smallest integer such that
\begin{equation}\label{eq:aaa1}
 R^{-k_{\delta}} < \delta.
 \end{equation}
We also note that one must have $R^{-k_{\delta} +1} \geq \delta$} since $k_\delta$ is the smallest integer fulfilling \eqref{eq:aaa1}. With the help of (\ref{eq:dualp4}), we have
\begin{equation}
  \begin{split}
    \mathbf{E}_{\delta}(\bu_{\delta}) & \geq \mathbf{J}_{\delta}(\bv_{\delta},\bm{\psi}_{\delta})= \int \bff \cdot \bm{\psi}_{\delta} - \frac{\delta}{2} \mathbf{P}_{\lambda,\mu}(\bv_{\delta},\bv_{\delta}) - \frac{\delta}{2}\mathbf{P}_{\lambda,\mu}(\bm{\psi}_{\delta},\bm{\psi}_{\delta})   \\
      & \geq c_0 \gamma_{k_{\delta}} \tau_{\delta} q^{-k_{\delta}} R^{2k_{\delta}} -C_{\lambda,\mu} \delta^{-1} \tau_{\delta}^2 k_{\delta} - C_{\lambda,\mu} \delta  \tau_{\delta}^2 k_{\delta} R^{2 k_{\delta}} \\
      & \geq \tau_{\delta} R^{k_{\delta}}  \left(  c_0 \gamma_{k_{\delta}} \left(\frac{R}{q} \right)^{k_{\delta}} - C_{\lambda,\mu} \frac{1}{(\delta R^{k_{\delta}} )} \tau_{\delta} k_{\delta} - C_{\lambda,\mu} \tau_{\delta} k_{\delta}(\delta R^{k_{\delta}})   \right).
  \end{split}
\end{equation}
The choice of $R^{k_{\delta}} < \delta$ with $1< \delta R^{k_{\delta}} \leq R$ ensures that the last two contributions are of comparable order. We then find, for some $C_{\lambda,\mu}>0$,
\begin{equation}\label{eq:dd1}
  \mathbf{E}_{\delta}(\bu_{\delta}) \geq \tau_{\delta} R^{k_{\delta}}  \left(  c_0 \gamma_{k_{\delta}} \left(\frac{R}{q} \right)^{k_{\delta}}  - C_{\lambda,\mu} \tau_{\delta} k_{\delta}  \right).
\end{equation}
We choose $\tau_{\delta}$ to be
\begin{equation}\label{eq:dd2}
  \tau_{\delta}= \frac{1}{2 C_{\lambda,\mu} k_{\delta}  } c_0 \gamma_{k_{\delta}} \left(\frac{R}{q} \right)^{k_{\delta}},
\end{equation}
and then from \eqref{eq:dd1} and \eqref{eq:dd2} we readily have that
\begin{equation}\label{eq:dd3}
  \mathbf{E}_{\delta}(\bu_{\delta}) \geq \tau_{\delta} R^{k_{\delta}} \left( \frac{1}{2 C_{\lambda,\mu} k_{\delta}  } c_0 \gamma_{k_{\delta}} \left(\frac{R}{q} \right)^{k_{\delta}}  \right) = \frac{1}{4 C_{\lambda,\mu} k_{\delta} } (c_0 \gamma_{k_{\delta}})^2 \left(\frac{R^3}{q^2} \right)^{k_{\delta}}.
\end{equation}
By the assumption, $q>R^*$ and if the sequence of the Fourier coefficients $\gamma_k$ of the force term $\mathbf{f}$ decays not very quickly (ensuring that the RHS term of \eqref{eq:dd3} goes to infinity as $\delta\rightarrow+0$), we easily see from \eqref{eq:dd3} that $\mathbf{E}_{\delta}(\bu_{\delta}) \rightarrow +\infty$ as $\delta \rightarrow +0$.

This proof is complete.
\end{proof}

\begin{rem}\label{rem:modification2}
Similar to Remark~\ref{rem:modification1}, if one chooses the plasmone constant $c$ in (\ref{eq:tensor2}) to be
\[
c:= -\frac{\lambda+3\mu}{\lambda+\mu},
\]
then by following a completely similar argument to the proof of Theorem~\ref{thm:main2}, one can show that if the Fourier coefficients $\beta_k$ of the force term $\bff$ decays not very quickly, then the anomalous localized resonance occurs.
\end{rem}

\subsection{Non-resonance in the radial case}

In Section~\ref{sect:ALR1}, we show that for certain source/force terms lying within the critical radius $R^*$, the resonance occurs. In this section, we shall show that for a certain source term lying outside the critical radius, then resonance does not occur. To that end, we would consider our study in the radial geometry by assuming that the core $\Sigma=B_1$.

\begin{thm}\label{thm:main3}
Consider the elastic configuration $(\mathbf{C}_{\widetilde\lambda,\widetilde\mu},\bff)$, where $\mathbf{C}_{\widetilde\lambda,\widetilde\mu}$ is described in \eqref{eq:tensor1}--\eqref{eq:tensor2} with $c$ given in \eqref{eq:tensor3} and, $\Omega=B_R$ for a certain $R>1$ and $\Sigma= B_1$ .
 Consider the elastostatic system (\ref{eq:lame1}), with $\mathbf{C}_{\widetilde\lambda,\widetilde\mu}$ describe above. Let the source $\bff$ be given by (\ref{eq:source2}) with $\{\beta_k\} \in l^2(\mathbb{N},\mathbb{R})$, and
\begin{equation}\label{coeff_f_relation}
  \frac{\beta_k}{\gamma_{k-2}} =-\frac{ q^2 (\lambda + 3\mu)}{(k-2) (\lambda + \mu) (q^2-R^2) } \quad \quad k>2; \quad \beta_k=0, \quad k=1,2.
\end{equation}
Then for any $q>R^* := R^{3/2}$, the configuration $(\mathbf{C}_{\widetilde\lambda,\widetilde\mu},\bff)$ is non-resonant.
\end{thm}

\begin{proof}
We make use of the primal variational principle to show the non-resonance result. We shall construct the test function $(\bv_{\delta},\bw_{\delta})$, satisfying the constraint
\begin{equation}\label{eq:thm_constraint1}
   A\cl_{\lambda,\mu} \bv_{\delta} -\cl_{\lambda,\mu} \bw_{\delta} =\bff
\end{equation}
such that the energy along this sequence , $\mathbf{I}_{\delta}(\bv_{\delta},\bw_{\delta})$ remains bounded. To that end, our strategy is to decompose the source $\bff$ into a low-frequency part and a high-frequency part as that
\begin{equation}\label{f_decompose}
\begin{split}
\bff=&\bff^{\text{low}}+\bff^{\text{high}},\\
\bff^{\text{low}}:=\sum_{k=3}^{k^*} (\beta_k \bff_{1,k}^q + \gamma_k \bff_{2,k}^q)&, \quad \bff^{\text{high}}:=\sum_{k=k^*+1}^{\infty}  (\beta_k \bff_{1,k}^q + \gamma_k \bff_{2,k}^q),
\end{split}
\end{equation}
where $\bff_{1,k}^q$ and $\bff_{2,k}^q$ are given by (\ref{eq:fkq1}) and (\ref{eq:fkq2}), respectively, and $k^*$ will be chosen to depend on $\delta$. Indeed, we shall choose $k^*=k^*(\delta)$ to be the smallest integer such that $R^{-k^*} > \delta$, and this will be explicitly specified again in what follows. We then construct $(\bv_{\delta},\bw_{\delta})$ with $\bv_{\delta} =\bv_{\delta}^{\text{low}}  + \bv_{\delta}^{\text{high}}$ and $\bw_{\delta}=\bw_\delta^{\text{high}}$ as follows:
\begin{equation}\label{condi:v_low}
  \bv_{\delta}^{\text{low}} \quad \text{satisfies} \quad  A\cl_{\lambda,\mu} \bv_{\delta}^{\text{low}} =\bff^{\text{low}},
\end{equation}
\begin{equation}\label{condi:v_high}
  \bv_{\delta}^{\text{high}} \quad \text{satisfies} \quad  A\cl_{\lambda,\mu} \bv_{\delta}^{\text{high}}|_{\partial B_q(0)} =\bff^{\text{high}},
\end{equation}
\begin{equation}\label{condi:w_high}
 \bw_{\delta}^{\text{high}} \quad \text{satisfies} \quad    -\cl_{\lambda,\mu} \bw_{\delta}^{\text{high}} = -A\cl_{\lambda,\mu} \bv_{\delta}^{\text{high}} +\bff^{\text{high}}.
\end{equation}
This construction yields $(\bv_{\delta},\bw_{\delta})$, which satisfies the constraint (\ref{eq:thm_constraint1}) of the primal problem (\ref{eq:primal}). Furthermore, we shall show that with an appropriate choice of the cutoff integer $k^*=k^*(\delta)$ in (\ref{f_decompose}), $\mathbf{I}_{\delta}(\bv_{\delta},\bw_{\delta})$ remains bounded as $\delta \rightarrow +0$.

Next, we construct $\bv_{\delta}^{\text{low}}$. First, we present the base function $\hat{\bv}_{k}$ for our construction, which can be represented as follows:
\begin{equation}\label{eq:base1}
\hat{\bv}_k(x)=\begin{cases}
\hat{\bv}_k^{(c)}(x),\quad & r\leq 1,\\
\hat{\bv}_k^{(s)}(x),\quad & 1<r\leq R,\\
\hat{\bv}_k^{(f)}(x),\quad & R<r\leq q,\\
\hat{\bv}_k^{(e)}(x),\quad & r>q,
\end{cases}\qquad k=3,4,5,\ldots,
\end{equation}
with
\begin{align}
&\hat{\bv}_k^{(c)}(x):= \left[
  \begin{array}{c}
     r^k \cos(k \theta) - k \alpha (r^2-1)  r^{k-2} \cos((k-2) \theta)   \\
     r^k \sin(k \theta) + k \alpha (r^2-1)  r^{k-2} \sin((k-2) \theta) \\
  \end{array}
\right],
\quad \quad r \leq 1 , \\
&\hat{\bv}_k^{(s)}(x):= \left[
  \begin{array}{c}
    r^{-k} \cos(k \theta)  \\
    r^{-k} \sin(k \theta)   \\
  \end{array}
\right],
\quad  \quad 1<r \leq R,  \\
&\hat{\bv}_k^{(f)}(x):= R^{-2k} \left[
  \begin{array}{c}
     r^k \cos(k \theta) - k \alpha (r^2-R^2)  r^{k-2} \cos((k-2) \theta)  \\
     r^k \sin(k \theta) + k \alpha (r^2-R^2)  r^{k-2} \sin((k-2) \theta) \\
  \end{array}
\right],
\ R<r \leq q,  \\
&\hat{\bv}_k^{(e)}(x):= \left( \frac{q}{R} \right)^{2 k} \bigg(
\left[
  \begin{array}{c}
   r^{-k} \cos(k \theta) \\
   r^{-k} \sin(k \theta) \\
  \end{array}
\right]\nonumber\\
&+c_1 k
\left[
  \begin{array}{c}
   (k-2) \alpha (r^2 - q^2) r^{-k} \cos(k \theta) +   r^{-(k-2)} \cos((k-2) \theta) \\
   (k-2) \alpha (r^2 - q^2) r^{-k} \sin(k \theta) -   r^{-(k-2)} \sin((k-2) \theta) \\
  \end{array}
\right]
\bigg),\ r>q,
\end{align}
where 
\[
c_1 = \frac{ \alpha (R^2 - q^2)}{q^4}
\] 
and $\alpha$ is given in (\ref{coeff_alpha}). We note that $\hat{\bv}_{k} $ have the following properties:\smallskip
\begin{enumerate}
  \item $\hat{\bv}_{k} $ is continuous on all $\mathbb{R}^2$;\smallskip
  \item $\hat{\bv}_{k} $ satisfies $A\cl_{\lambda,\mu} \hat{\bv}_{k} =0 $ for $x\in \mathbb{R}^2 \backslash \partial B_q$;\smallskip
  \item Along $\partial B_q$, $\hat{\bv}_{k} $ has a jump in its normal flux:
\begin{equation}
\begin{split}
  \left[ \frac{\partial \hat{\bv}_{k} }{\partial \bm{\nu}} \right]_{\partial B_q} = & -c_2 k q^k R^{-2k} \left( q^2 (\lambda + 3\mu)\right)
\left[
  \begin{array}{c}
    \cos(k \theta) \\
    \sin(k \theta) \\
  \end{array}
\right] \\
    & + c_2 k q^k R^{-2k} \left( (k-2) (\lambda + \mu) (q^2-R^2) \right)
\left[
  \begin{array}{c}
    \cos((k-2) \theta) \\
   -\sin((k-2) \theta) \\
  \end{array}
\right],
\end{split}
\end{equation}
where $[\cdot]$ denote the jump of the normal flux and
\begin{equation} \label{coeff_c_3}
  c_2= \frac{4 \mu (\lambda + 2\mu)}{q^3 (\lambda + 3 \mu)^2} .
\end{equation}
\end{enumerate}
By using the properties listed above, it is easy to verify that with an appropriate constant multiple $\tau_k \hat{\bv}_{k}$, one has
\begin{equation}\label{eq:ppp1}
 A\cl_{\lambda,\mu} (\tau_k \hat{\bv}_{k}) |_{\partial B_q} = \beta_k \bff_{1,k}^q + \gamma_{k-2} \bff_{2,k-2}^q.
\end{equation}
Next we choose $\tau_k$ such that
\begin{equation}\label{eq:ppp2}
\tau_k \cdot \big[ -c_2 k q^k R^{-2k} \left( q^2 (\lambda + 3\mu)\right) \big] = \beta_k,
\end{equation}
and
\begin{equation}\label{eq:ppp3}
\tau_k \cdot \big[ c_2 k q^k R^{-2k} \left( (k-2) (\lambda + \mu) (q^2-R^2) \right) \big]= \gamma_{k-2}.
\end{equation}
With the help of (\ref{coeff_f_relation}), one has by direct calculations that
\begin{equation}\label{eq:ppp4}
\tau_k:= \frac{\beta_k}{-c_2 k  \left( q^2 (\lambda + 3\mu)\right)} q^{-k} R^{2k}.
\end{equation}
Now we set
\begin{equation}\label{eq:v_low}
\bv_{\delta}^{\text{low}}:=\sum_{k=3}^{k^*} \tau_k \hat{\bv}_{k},
\end{equation}
with $\tau_k$ given in \eqref{eq:ppp4}. By combining \eqref{eq:ppp1}--\eqref{eq:ppp4}, along with straightforward calculations, one can directly verify that $\bv_\delta^{\text{low}}$ defined in \eqref{eq:v_low} satisfies (\ref{condi:v_low}). 

After constructing $\bv_{\delta}^{\text{low}}$, we next give the construction of $\bv_{\delta}^{\text{high}}$ and $\bw_{\delta}^{\text{high}}$, respectively, in \eqref{condi:v_high} and \eqref{condi:w_high}. Similar to the construction of $\bv_\delta^{\text{low}}$ via a certain based function $\hat{\bv}_k$ in \eqref{eq:base1}. The construction of the function $\bv_{\delta}^{\text{high}}$ shall also be constructed from certain base functions $\widehat{\bV}_k$ for $k=3,4,5,\ldots$. Those functions do not fulfil $-A\cl_{\lambda,\mu} \bv =0$ on $\partial B_1$ or $\partial B_R$, but they are small along these curves. We introduce $\widehat{\bV}_k$ as follows:
\begin{equation}\label{eq:base2}
\widehat{\bV}_k(x)=\begin{cases}
\widehat{\bV}_k^{(i)}(x),\quad & r\leq q,\\
\widehat{\bV}_k^{(o)}(x),\quad & r>q,
\end{cases}\qquad\qquad k=3,4,5,\ldots,
\end{equation}
with
\begin{align}
& \widehat{\bV}_k^{(i)}(x):= \frac{c_3}{k} \left[
  \begin{array}{c}
     r^k \cos(k \theta) - k \alpha (r^2-q^2)  r^{k-2} \cos((k-2) \theta)  \\
     r^k \sin(k \theta) + k \alpha (r^2-q^2)  r^{k-2} \sin((k-2) \theta) \\
  \end{array}
\right]\nonumber\\
& \qquad\qquad\qquad+c_4 \left[
  \begin{array}{c}
     r^{k-2} \cos((k-2) \theta)  \\
     -r^{k-2} \sin((k-2) \theta) \\
  \end{array}
\right]
\quad \quad r\leq q ; \\
& \widehat{\bV}_k^{(o)}(x):= c_4 q^{2(k-2)} \left[
  \begin{array}{c}
   r^{-(k-2)} \cos((k-2) \theta) + (k-2) \alpha (r^2 - q^2) r^{-k} \cos (k \theta) \\
  -r^{-(k-2)} \sin((k-2) \theta) + (k-2) \alpha (r^2 - q^2) r^{-k} \sin (k \theta) \\
  \end{array}
\right]\nonumber\\
&\qquad\qquad\qquad + \frac{c_3}{k} q^{2k} \left[
  \begin{array}{c}
   r^{-k} \cos(k \theta) \\
   r^{-k} \sin(k \theta) \\
  \end{array}
\right] \qquad  \qquad r>q;
\end{align}
where
\[
 c_3= -(\lambda + 3 \mu)\quad \mbox{and} \quad c_4 = (\lambda + \mu) (q^2-R^2).
 \]
Recall that $A(x)=1$ in a neighborhood of $\|x\|=q$. One can show by direct calculations that the jump of the normal flux of $\widehat{\bV}_k$ in \eqref{eq:base2} on $\partial B_q$ is given by
\begin{equation}\label{eq:qqq1}
  \begin{split}
  \left[ \frac{\partial \hat{\bV}_{k} }{\partial \bm{\nu}} \right]_{\partial B_q(0)} = & -c_5 q^k \left( q^2 (\lambda + 3\mu)\right)
\left[
  \begin{array}{c}
    \cos(k \theta) \\
    \sin(k \theta) \\
  \end{array}
\right] \\
    & + c_5  q^k  \left( (k-2) (\lambda + \mu) (q^2-R^2) \right)
\left[
  \begin{array}{c}
    \cos((k-2) \theta) \\
   -\sin((k-2) \theta) \\
  \end{array}
\right]
\end{split}
\end{equation}
where
\[
c_5 = -\frac{4 \mu (\lambda + 2\mu)}{q^3 (\lambda + 3 \mu)}.
\]
Therefore if we set
\begin{equation}\label{eq:v_high}
 \bv_{\delta}^{\text{high}} := \sum_{k>k^*} \tau_k \hat{\bV}_k, \quad \tau_k= \frac{\beta_k}{-c_5 \left( q^2 (\lambda + 3\mu)\right)} q^{-k},
\end{equation}
then by using (\ref{coeff_f_relation}) and \eqref{eq:qqq1}, one can show by direct calculations that (\ref{condi:v_high}) is satisfied:
\begin{equation}\label{eq:qqq2}
  A\cl_{\lambda,\mu} \bv_{\delta}^{\text{high}}|_{\partial B_q} =\bff^{\text{high}}. 
\end{equation}
We emphasize that $\bv_{\delta}^{\text{high}}$ is not a solution to \eqref{eq:qqq2} on all of $\mathbb{R}^2$ due to that the normal fluxes at $\|x\|=1$ and $\|x\|=R$. In order to construct a solution to \eqref{eq:qqq2} on $\mathbb{R}^2$, we introduce $\bw_{\delta}^{\text{high}}$ as follows:
\begin{equation}\label{eq:qqq3}
  \begin{split}
   -\cl_{\lambda,\mu} \bw_{\delta}^{\text{high}} & = -A\cl_{\lambda,\mu} \bv_{\delta}^{\text{high}} +\bff^{\text{high}}\\
      & = -\sum_{k>k^*} \tau_k \left[ A \frac{\partial \bv_{\delta}^{\text{high}} }{\partial \bm{\nu}} \right]_{\partial B_1} -\sum_{k>k^*} \tau_k \left[ A \frac{\partial \bv_{\delta}^{\text{high}} }{\partial \bm{\nu}} \right]_{\partial B_R}
  \end{split}
\end{equation}
Clearly, $\bw_\delta^{\text{high}}$ satisfies \eqref{condi:w_high}.  With the test functions being ready, it remains to calculate the energy $\mathbf{I}_{\delta}(\bv_{\delta},\bw_{\delta})$, for the choice of $\bv_{\delta} =\bv_{\delta}^{\text{low}}  + \bv_{\delta}^{\text{high}}$ and $\bw_{\delta}=\bw_\delta^{\text{high}}$. In this step, we shall choose an appropriate cut-off frequency, $k^*=k^*(\delta)$ to ensure that $\mathbf{I}_{\delta}(\bv_{\delta},\bw_{\delta})$ remains uniformly bounded as $\delta \rightarrow +0$.

We first have by direct calculations that
\begin{equation}\label{energy_low_1}
\delta \mathbf{P}_{\lambda,\mu}(\bv_{\delta}^{\text{low}}, \bv_{\delta}^{\text{low}}) \leq C_{\lambda, \mu} \delta \sum_{k\leq k^*} |\beta_k|^2 \left( \frac{R^2}{q} \right)^{2k} \text{max}\left\{1, \left( \frac{q}{R^2} \right)^k \right\}^2.
\end{equation}
For the case when $q\geq R^2$, we easily have from \eqref{energy_low_1} that
\begin{equation}\label{eq:rrr1}
\delta \mathbf{P}_{\lambda,\mu}(\bv_{\delta}^{\text{low}}, \bv_{\delta}^{\text{low}}) \leq C_{\lambda, \mu} \delta \sum_{k\leq k^*} |\beta_k|^2 \leq C_{\lambda, \mu} \delta,
\end{equation}
which is obvious bounded. The other case when $R^* <q<R^2$ is much more subtle. It is noted that the estimate (\ref{energy_low_1}) can be simplified in the case when $R^* <q<R^2$ to be
\begin{equation}\label{energy_control_v_low}
\delta \mathbf{P}_{\lambda,\mu}(\bv_{\delta}^{\text{low}}, \bv_{\delta}^{\text{low}}) \leq C_{\lambda, \mu} \delta \sum_{k\leq k^*} |\beta_k|^2 \left( \frac{R^2}{q} \right)^{2k^*} .
\end{equation}
In what follows, we shall show that with a special choice of $k^*$, the RHS of \eqref{energy_control_v_low} can be bounded. 
The energy of $\bv_{\delta}^{\text{high}}$ is easy to control:
\begin{equation}\label{eq:rrr2}
 \delta \mathbf{P}_{\lambda,\mu}(\bv_{\delta}^{\text{high}}, \bv_{\delta}^{\text{high}}) \leq C_{\lambda, \mu} \delta \sum_{k} |\beta_k|^2 \leq C_{\lambda, \mu}. 
\end{equation}
Next, we study the energy due to $\bw_{\delta}$, and with the help of (\ref{eq:v_high}), along with the standard elliptic estimates, we have
\begin{equation}\label{energy_control_w_high}
  \begin{split}
    &  \frac{1}{\delta} \mathbf{P}_{\lambda,\mu}(\bw_{\delta}, \bw_{\delta}) \leq C \frac{1}{\delta} \|  -A\cl_{\lambda,\mu} \bv_{\delta}^{\text{high}} +\bff^{\text{high}} \|^2_{H^{-1}} \\
      & \leq C \frac{1}{\delta} \sum_{k>k^*} |\tau_k|^2 R^{2k} k \leq C \sum_{k>k^*} |\beta_k|^2  \frac{1}{\delta} \left( \frac{R}{q} \right)^{2k^*}
  \end{split}
\end{equation}
Balancing the right hand sides of the bounds in (\ref{energy_control_v_low}) and (\ref{energy_control_w_high}), we choose $k^*$ so that
\begin{equation}
  \delta \left( \frac{R^2}{q} \right)^{2k^*} \thicksim   \frac{1}{\delta} \left( \frac{R}{q} \right)^{2k^*};
\end{equation}
namely we choose $k^*=k^*(\delta)$ to be the smallest integer with $R^{-k^*}<\delta$ such that
\begin{equation}\label{choosing_k}
  \delta \leq R^{-k^*+1} \quad \text{and} \quad \frac{1}{\delta} \leq R^{k^*}.
\end{equation}
Combining (\ref{choosing_k}) with (\ref{energy_control_v_low}) and (\ref{energy_control_w_high}), we obtain
\begin{equation}
    \delta \mathbf{P}_{\lambda,\mu}(\bv_{\delta}^{\text{low}}, \bv_{\delta}^{\text{low}}) \leq C_{\lambda, \mu} \sum_{k\leq k^*} |\beta_k|^2 \left( \frac{R^3}{q^2} \right)^{k^*(\delta)};
\end{equation}
and
\begin{equation}
   \frac{1}{\delta} \mathbf{P}_{\lambda,\mu}(\bw_{\delta}, \bw_{\delta}) \leq C \sum_{k>k^*} |\beta_k|^2  \left( \frac{R^3}{q^2} \right)^{k^*(\delta)}.
\end{equation}
Hence, if $q>R^* = R^{3/2}$, $\mathbf{I}_{\delta}(\bv_{\delta}, \bw_{\delta})$ is bounded as $\delta \rightarrow +0$. Therefore, by \eqref{eq:primalp6}, the elastic configuration is non-resonant.

The proof  is complete.

\end{proof}

\section{Are there perfect plasmon waves in three dimensions?}\label{sect:4d}

As can be seen from our earlier study, the perfect plasmone elastic waves in Lemma~\ref{lem:perfectwaves} play a critical role in establishing the resonance and non-resonance results. Hence, if one intends to extend the resonance and non-resonance results to the three-dimensional case, it would be natural to ask whether there exist perfect plasmone elastic waves in $\mathbb{R}^3$. First of all, we consider the base solutions to the equation, $\mathcal{L}^{\lambda, \mu} \mathbf{u} = 0$, in three dimensions and they are given as follows:
\begin{align}
& \mathbf{M}_n(x):=\curl \{ x r^n Y_n(\hat{x})\};  \quad  \curl\, \mathbf{M}_n(x); \quad \nabla r^n Y_n(\hat{x});\\
& \mathbf{N}_n(x):=\curl\{ x r^{-n-1} Y_n(\hat{x})\};  \quad  \curl\, \mathbf{N}_n(x); \quad \nabla r^{-n-1} Y_n(\hat{x});
\end{align}
where $\hat x:=x/\|x\|\in\mathbb{S}^2$ for $x\in\mathbb{R}^3\backslash\{0\}$; and $Y_n(x)$ is the spherical harmonics of order $n$ for $n\in\mathbb{N}\cup\{0\}$. 
By direct calculations, we have
\begin{align}
 \mathbf{M}_n(x)=& r^n \text{Grad}\, Y_n(\hat{x})\times \hat{x},\\
\mathbf{N}_n(x)=& r^{-n-1} \text{Grad}\, Y_n(\hat{x})\times \hat{x}. 
\end{align}
Set
\begin{equation}
  \mathbf{B}=\text{Grad}\, Y_n(\hat{x})\times \hat{x} ,
\end{equation}
and the first components of $\frac{\partial \mathbf{M}_n} {\partial \bm{\nu}}$, $\frac{\partial \mathbf{N}_n} {\partial \bm{\nu}}$ are, respectively, given as follows:
\begin{equation}
 \bigg(\frac{\partial \mathbf{M}_n} {\partial \bm{\nu}}\bigg)_{  1 }= \sum_{j=1}^3 \mu \bm{\nu}_j \left( n r^{n-2} x_1 \mathbf{B}_j + r^n \frac{\partial \mathbf{B}_j}{\partial x_1}   \right) + \mu n r^{n-1} \mathbf{B}_1,
\end{equation}
\begin{equation}
 \bigg( \frac{\partial \mathbf{N}_n} {\partial \bm{\nu}} \bigg)_1 = \sum_{j=1}^3 \mu \bm{\nu}_j \left( (-n-1) r^{-n-3} x_1 \mathbf{B}_j + r^{-n-1} \frac{\partial \mathbf{B}_j}{\partial x_1}   \right) + \mu (-n-1) r^{-n-2} \mathbf{B}_1,
\end{equation}
where $\mathbf{B}_j$ is the $j$th-component of $\mathbf{B}$. Hence, there does not exist a plasmone constant $c$ such that
\begin{equation}
 c  \frac{\partial \mathbf{M}_n} {\partial \bm{\nu}}  =  \frac{\partial \mathbf{N}_n} {\partial \bm{\nu}}  .
\end{equation}
Therefore, it seems that there are no perfect plasmon waves in three dimensions. We shall address this issue in our future investigation.

\section*{Acknowledgement}

The work was supported by the FRG grants from Hong Kong Baptist University, Hong Kong RGC General Research Funds, 12302415 and 405513, and the NSF grant of China, No. 11371115.

\end{document}